\documentclass[12pt,a4paper,oneside]{amsart}
\usepackage[a4paper]{geometry}
\usepackage{fouriernc} 

\usepackage{mathrsfs}

\newtheorem{theorem}{Theorem}[section]
\newtheorem*{A}{Theorem FA}
\newtheorem*{B}{Theorem FB}
\newtheorem*{C}{Theorem FC}
\newtheorem*{LL}{Theorem LL}
\newtheorem{lemma}[theorem]{Lemma}
\newtheorem{corollary}[theorem]{Corollary}
\newtheorem{proposition}[theorem]{Proposition}

\theoremstyle{remark}
\newtheorem*{remark}{Remark}

\newcommand{\set}[2]{\ensuremath{\{ #1 \>|\> #2 \}}}

\def\form{\ensuremath{(\,\cdot\, , \cdot\,)}}

\begin{document}

\title{On $\delta$-derivations of Lie algebras and superalgebras}

\author{Pasha Zusmanovich}
\address{}
\email{pasha.zusmanovich@gmail.com}

\date{last minor revision January 18, 2019}
\thanks{J. Algebra \textbf{324} (2010), no.12, 3470--3486; arXiv:0907.2034}

\begin{abstract}
We study $\delta$-derivations -- a construction simultaneously generalizing 
derivations and centroid. First, we compute $\delta$-derivations of current Lie algebras and of modular Zassenhaus 
algebra. This enables us to provide examples of Lie algebras having 
$\frac{1}{2}$-derivations which are divisors of zero, thus answering negatively a 
question of Filippov.
Second, we note that $\delta$-derivations allow, in some circumstances,
to construct examples of non-semigroup gradings of Lie algebras, 
in addition to the recent ones discovered by Elduque.
Third, we note that utilizing the construction of the Grassmann envelope allows to obtain results 
about $\delta$-(super)derivations of Lie superalgebras from the corresponding results about Lie algebras.
In this way, we prove that prime Lie superalgebras do not possess nontrivial 
$\delta$-(super)derivations, generalizing the recent result of Kaygorodov.
\end{abstract}

\maketitle

\section*{Introduction}

Let $L$ be a Lie algebra and $\delta$ is an element of the ground field.
A linear map $D: L \to L$ is called a \textit{$\delta$-derivation} of $L$ if
\begin{equation}\label{delta-der}
D([x,y]) = \delta [D(x),y] + \delta [x,D(y)]
\end{equation}
for any $x,y\in L$.
It is clear that ordinary derivations are $1$-derivations,
the elements of the centroid of the algebra are $\frac{1}{2}$-derivations,
and the maps vanishing on the commutant of the algebra are $0$-derivations.
Thus, $\delta$-derivations appear to be a natural simultaneous generalization of the
notions of derivations and centroid.
This alone justifies their study, but there is more to that:
$(-1)$-derivations arise in the study
of some noncommutative Jordan algebras (see \cite{hopkins} and references therein), 
as well as are closely related to 
commutative $2$-cocycles (the same way as ordinary derivations are related to the second 
cohomology group; see \cite{comm2}), and $\delta$-derivations arise in description of 
(ordinary) derivations of certain current Lie algebras (\cite[\S 3]{without-unit}).
$(-1)$-derivations (under the name of \textit{antiderivations}) are mentioned also in 
the classical book of Jacobson (\cite[p. 179]{jacobson}).

$\delta$-derivations were studied by Filippov in a series of papers (\cite{filippov-1},
\cite{filippov-2} and \cite{filippov-ass}). 
In particular, he proved that prime Lie algebras, as a rule, 
do not have nontrivial (i.e., different from the above three examples for 
$\delta = 0, 1, \frac{1}{2}$) $\delta$-derivations.
He also noted that $\frac{1}{2}$-derivations of a prime Lie algebra form
a commutative ring which contains the centroid of the algebra,
and posed the question: is it true that this ring does not contain divisors of 
zero?

In the recent papers \cite{kayg-first,kayg}, Kaygorodov initiated a study of 
a similar notion for nonassociative superalgebras. 
In particular, he proved in \cite{kayg} that finite-dimensional classical Lie 
superalgebras over a field of characteristic zero do not have nontrivial 
$\delta$-derivations.

A more general concept -- the so-called \textit{quasiderivations} -- was studied
by Leger and Luks in \cite{leger-luks}. These are the linear maps $D: L\to L$ of a 
Lie algebra $L$ such that there is a linear map $F(D):L\to L$ such that 
$$
F(D)([x,y]) = [D(x),y] + [x,D(y)]
$$
holds for any $x,y\in L$. 
Clearly, $\delta$-derivations with $\delta\ne 0$ are quasiderivations with 
${F(D) = \frac{1}{\delta}D}$.
While Filippov's methods are essentially ring-theoretic and consist of sophisticated 
manipulation with identities, Leger and Luks confined themselves to the case
of finite-dimensional algebras over an algebraically closed field and used a
Lie-algebraic technique of the root space decomposition with respect to a torus.
In this way, they obtained some useful criteria for quasiderivations of a Lie
algebra to be reduced to the ordinary derivations and centroid.

We hope that the present paper reveals further interesting interconnections
and provides additional motivation to study $\delta$-derivations.
In \S 1 we give a simple
formula expressing $\delta$-derivations of current Lie algebras in terms of its 
tensor factors, similar to the known formula for the ordinary derivations.
In \S 2 we compute $\delta$-derivations of the modular Zassenhaus algebra. 
It turns out that Zassenhaus algebra has $\frac{1}{2}$-derivations
which are divisors of zero, what provides a negative answer to the Filippov's question.
Further examples can be obtained by tensoring
Zassenhaus algebra with the reduced polynomial algebra and adding a ``tail'' of derivations 
-- a construction typical in the structure theory of modular Lie algebras. 
Note that the property to possess nontrivial $\delta$-derivations distinguishes Zassenhaus 
algebra, together with $sl(2)$, among all simple finite-dimensional Lie algebras.

In \S 3 we note an elementary, albeit somewhat unexpected connection with a different topic --
non-semigroup gradings. The relevant story is quite intriguing: in the influential
paper \cite{zass}, Patera and Zassenhaus claimed that each grading
of a finite-dimensional Lie algebra is a semigroup grading. This was believed to be true
for almost two decades, until Elduque found a flaw in the proof and constructed 
counterexamples to this statement (\cite{elduque-1} and \cite{elduque-2}).
It turns out that, under certain additional conditions, $\delta$-derivations
lead to non-semigroup gradings of Lie algebras, what allows to construct
further counterexamples systematically.
We also note a restriction on $\delta$ which should satisfy a non-nilpotent 
$\delta$-derivation of a complex perfect finite-dimensional Lie algebra.

Finally, in \S 4, we show that prime Lie superalgebras do not have 
nontrivial $\delta$-derivations and $\delta$-superderivations. 
This extends the result of \cite{kayg} mentioned above.
Kaygorodov achieves this by utilizing the known classification of simple classical
Lie superalgebras and case-by-case calculations, while our approach based
on passing from superalgebras to algebras via the venerable Grassmann envelope,
and utilizing Filippov's results.

In the last section we take the liberty to suggest some directions for further investigations.

\section*{Notation and conventions}

Throughout the paper, $\delta$ denotes an element of the ground field $K$, which is assumed
to be of characteristic $\ne 2,3$ (though some of the intermediate results are valid
also in characteristic $3$).

It is obvious that for a fixed $\delta$, the set of all $\delta$-derivations of a given 
Lie algebra $L$ forms a linear space, which will be denoted by $Der_\delta(L)$.
For $\delta=1$, the space (actually, a Lie algebra) of the ordinary derivations
will retain the usual notation $Der(L)$.

In what follows, we will need also a straightforward extension of this notion
by considering $\delta$-derivations with values in modules.
Let $L$ be a Lie algebra and $M$ is an $L$-module.
A linear map $D: L \to M$ will be called a \textit{$\delta$-derivation of $L$ with values in $M$}
if
$$
D([x,y]) = \delta y \bullet D(x) - \delta x \bullet D(y) 
$$
for any $x,y\in L$.
The linear space of all such maps will be denoted as $Der_\delta (L,M)$.
The same way as for $\delta$-derivations with values in the adjoint module, 
in the case $\delta = 1$ the subscript $\delta$ will be omitted.

Other standard notions and notation we use: $Hom(V,W)$ is the space of all linear maps between
vector spaces $V$ and $W$. 
The commutant and the center of a Lie algebra $L$ are denoted as $[L,L]$ and $Z(L)$,
respectively.
A Lie algebra is called \textit{perfect} if it coincides with its commutant.
An algebra is called \textit{prime} if the product of any two its nonzero ideals
is nonzero.
The \textit{centroid} of an algebra $A$ is a set of linear maps $\chi\in Hom(A,A)$
commuting with all left and right multiplications in $A$: 
$\chi(ab) = \chi(a)b = a\chi(b)$ for any $a,b\in A$.
For an associative commutative algebra $A$, the operator of multiplication by an
element $a\in A$ is denoted as $R_a$.

For convenience, we record Filippov's results which we will cite frequently
(and generalize the first two of them to superalgebras in \S \ref{super}):

\begin{A}[\cite{filippov-2}, Theorem 2]
A prime Lie algebra does not have nonzero $\delta$-derivations if 
$\delta \ne -1, 0, \frac{1}{2}, 1$.
\end{A}

\begin{B}[\cite{filippov-1}, Corollary 3]
The space of $\frac{1}{2}$-derivations of a central prime Lie algebra having a 
nondegenerate symmetric bilinear invariant form, coincides with $K$. 
\end{B}

\begin{C}[\cite{filippov-2}, Corollary 1]
A central simple Lie algebra having nonzero $(-1)$-de\-ri\-va\-ti\-ons is a form of
$sl(2)$.
\end{C}

Filippov's results overlap with those of Leger--Luks:

\begin{LL}[Leger--Luks]
The space of quasiderivations of a simple finite-dimensional Lie algebra of rank $>1$
coincides with the direct sum of derivations and the centroid of the algebra.
\end{LL}

\begin{proof}
This is a direct consequence of \cite[Corollary 4.16]{leger-luks}. Note that the condition about 
``special weight spaces'' stated in \cite[Theorem 4.12]{leger-luks}, assumes that the rank 
of the algebra $>1$ (see \cite[Lemma 4.8(1)]{leger-luks}), and automatically follows from
the simplicity.
\end{proof}

\section{$\delta$-derivations of current Lie algebras}

Recall that the \textit{current Lie algebra} is a Lie algebra of the form $L\otimes A$
obtained by tensoring a Lie algebra $L$ with an associative commutative algebra
$A$, equipped with the Lie multiplication
\begin{equation}\label{curr}
[x\otimes a, y\otimes b] = [x,y]\otimes ab
\end{equation}
for $x,y\in L$, $a,b\in A$.
If $M$ is an $L$-module, and $V$ is an $A$-module, consider also the ``current module''
$M\otimes V$ with action
$$
(x \otimes a) \bullet (m \otimes v) = (x\bullet m) \otimes (a \bullet v)
$$
for $x\in L$, $a\in A$, $m\in M$, $v\in V$ (here, by abuse of notation, the same bullet 
sign denotes the respective actions of $L$, $A$ and $L\otimes A$).

\begin{theorem}\label{current}
Let $L$ be a Lie algebra, $A$ an associative commutative algebra with unit,
$M$ an $L$-module, $V$ a unital $A$-module, 
and either both $L$ and $M$, or both $A$ and $V$ are finite-dimensional.
Then there is an isomorphism of vector spaces:
\begin{align*}
Der_\delta (L &\otimes A, M \otimes V) \\ 
& \simeq Der_\delta (L,M) \otimes V \\ 
&\oplus 
\set{\varphi\in Hom(L,M)}{\varphi([x,y]) = \delta x \bullet \varphi(y) \text{ for any } x,y\in L} \otimes Der(A,V) \\ 
&\oplus Hom(L/[L,L], M^L) \otimes Hom(A,V) / (V + Der(A,V)).
\end{align*}
\end{theorem}

\begin{proof}
Verbatim repetition of the proof of 
\cite[Theorem 2.1]{low}\footnote[2]{
There is one minor inaccuracy in that proof, which does not affect the general flow
of reasoning and the final answer: namely, after substituting $b=1$ in the cocycle equation 
(2.2) (page 76 in the published version, lines 14-13 from the bottom), the conclusion about vanishing of 
$\varphi_i$'s with $i\in I_1$ and under assumption of the vanishing of the corresponding 
first tensor factor $x\bullet \varphi_i(y) - \varphi_i([x,y])$, is wrong. The correct 
conclusion is that corresponding $\varphi_i$'s are just ``very degenerate'':
$\varphi_i([L,L]) = 0$ and $\varphi_i(L) \subseteq M^L$, what gives rise to cocycles
of type (iii) in the statement of the theorem.
This is rectified in the arXiv version of that paper.
}.
\end{proof}

When taking $M=L$ and $V=A$, this generalizes the corresponding statement for ordinary 
derivations in \cite[Corollary 2.2]{low}, and numerous results about the centroid of 
some particular current Lie algebras scattered over the literature 
(\cite[Corollary 2.23]{benkart-neher}, 
\cite[Remark 2.19(1)]{gundogan}, \cite[Lemma 5.1]{krylyuk}, \cite[\S 3]{melville}, etc.).

\begin{corollary}\label{corr}
Let $L$ be a Lie algebra which is either perfect or centerless, 
$A$ is a commutative algebra with unit,
one of $L$, $A$ is finite-dimensional, and $\delta \ne 0,1$. Then
$Der_\delta (L\otimes A) \simeq Der_\delta(L) \otimes A$.
\end{corollary}

\begin{proof}
By assumption, the third direct summand in the isomorphism
of Theorem \ref{current} vanishes. Let us consider elements from the tensor factor 
of the second direct summand, i.e., linear maps $\varphi:L \to L$ satisfying the 
condition $\varphi([x,y]) = \delta [x,\varphi(y)]$ (and hence, by anti-commutativity,
also $\varphi([x,y]) = \delta [\varphi(x),y]$) for any $x,y\in L$.
For such maps we have, for any $x,y,z\in L$:
\begin{align*}
\varphi([[x,y],z]) &=
- \varphi([[z,x],y]) - \varphi([[y,z],x]) \\ &=
- \delta [\varphi([z,x]),y] - \delta [\varphi([y,z]),x] \\ &=
- \delta^2 [[\varphi(z),x],y] - \delta^2 [[y,\varphi(z)],x] \\ &=
\delta^2 [[x,y],\varphi(z)] \\ &= 
\delta \varphi([[x,y],z])
\end{align*}
(first and fourth equalities hold by the Jacobi identity).
As $\delta \ne 1$, this implies $\varphi([[L,L],L]) = 0$.
Hence if $L$ is perfect, all such $\varphi$'s vanish.
If $L$ is centerless, the equality
$$
0 = \varphi([[x,y],z]) = \delta [\varphi([x,y]),z] = \delta^2 [[\varphi(x),y],z]
$$
for any $x,y,z\in L$ implies $[\varphi(L),L] \subseteq Z(L) = 0$, and hence
$\varphi(L) \subseteq Z(L) = 0$. Thus the second direct summand in the 
isomorphism of Theorem \ref{current} vanishes too, and we are left with the first one.
\end{proof}

\section{$\frac{1}{2}$-derivations of the Zassenhaus algebra}\label{sect-zass}

Naturally, when one encounters a new Lie-algebraic invariant, one of the first 
examples one wishes to compute this invariant for, is finite-dimensional simple 
algebras. By Theorem LL, nontrivial $\delta$-derivations are possible only for algebras of rank $1$,
i.e., for $sl(2)$ and, in the case of the ground field of positive characteristic, for 
the Zassenhaus algebra $W_1(n)$ and the Hamiltonian algebra $H_2(m_1,m_2)$.
On the other hand, by Theorem FA, the only cases we may encounter are 
$\delta = -1, \frac{1}{2}, 0, 1$. The case $\delta = 1$ corresponds to ordinary
derivations which were studied exhaustively. The case $\delta = 0$ is not interesting at 
all: $0$-derivations are merely the maps vanishing on the commutant of an algebra 
(and thus, vanishing on the whole algebra in the simple case). 
By Theorem FC, the only central simple 
Lie algebras having nonzero $(-1)$-derivations are forms of
$sl(2)$. What happens in the remaining case $\delta = \frac{1}{2}$?

According to \cite[Theorem 5]{filippov-2}, the space of 
$\frac{1}{2}$-derivations of a prime Lie algebra $L$ forms an associative 
commutative algebra which contains the centroid of $L$. Moreover, by Theorem FB, for all simple Lie algebras having a nonzero symmetric
bilinear invariant form -- in particular, for all classical Lie algebras, as well as
for all Hamiltonian Lie algebras (see, for example,
\cite[\S 4.6, Theorem 6.5]{sf}) --
this commutative algebra coincides with the ground field (interpreted as operations
of multiplication by the field element).

Therefore, it remains only to consider $\frac{1}{2}$-derivations of the 
Zassenhaus algebra.

When computing outer derivations (and, more generally, cohomology) of a Lie 
algebra, we are greatly helped with the triviality of the Lie algebra 
(co)homology under the torus action -- a very simple, yet one of the most useful facts about (co)homology of Lie algebras.
Thus, when a Lie algebra in question possesses, for example, a grading 
compatible with the action of some torus, all outer derivations could be assumed to preserve 
that grading.

This nicety is, generally, no longer true in the case of 
$\delta$-de\-ri\-va\-tions. Moreover, $\delta$-derivations for $\delta \ne 1$ do not form a cohomology group: there are no notions of inner and outer $\delta$-derivations,
and we cannot take a quotient by an appropriate space of $1$-coboundaries.
Nevertheless, we can still benefit from considering the torus action, as was observed,
for the case of quasiderivations, by Leger and Luks in \cite[\S 4]{leger-luks}. 
Indeed, it is easy to check 
(as noted, for example, in \cite[\S 1]{filippov-2}) that for $\delta, \delta^\prime\in K$, 
$$
[Der_\delta(L), Der_{\delta^\prime}(L)] \subseteq Der_{\delta\delta^\prime}(L) ,
$$
where both $Der_\delta(L)$, $Der_{\delta^\prime}(L)$ are understood here as subspaces
of the Lie algebra \linebreak${Hom(L,L) \simeq gl(L)}$ of all linear maps of the Lie algebra $L$.
In particular, $Der_\delta(L)$ is invariant under inner derivations of $L$.
Let $T$ be a torus in $L$ such that $L$ decomposes into the direct sum of the root spaces
with respect to the action of $T$.
This action induces, in the standard way, the action of $T$ on 
$Hom(L,L) \simeq L \otimes L^*$, and the latter, as well as any its $T$-invariant 
subspace, can be decomposed into the direct sum of the root spaces with respect
to this action as well.

Let $G$ be an additively written abelian group, and $R$ a subset of $G$ which contains
$0$ and such that the sum of any two distinct elements of $R$ lies in $R$. By the 
\textit{Lie algebra of Witt type} $W_R$ we will understand
a Lie algebra having the basis $\set{e_\alpha}{\alpha\in R}$ with multiplication
\begin{equation*}
[e_\alpha, e_\beta] = (\beta - \alpha) e_{\alpha + \beta} ,
\end{equation*}
where $\alpha, \beta\in R$.

Specializing to $R = G = \mathbb Z_{p^n}$, for a prime $p$ and $n\in \mathbb N$,
we get the famous \textit{Zassenhaus algebra}
$W_1(n)$ (which is a $p^n$-dimensional algebra defined over a field of positive
characteristic $p$). The case $n=1$ deserves the special name of \textit{Witt algebra}. 
Specializing to $G = \mathbb Z$, and to $R = \mathbb Z$ and 
$R = \set{i\in \mathbb Z}{i \ge -1}$, we get 
the \textit{two-sided} and \textit{one-sided} (infinite-dimensional) 
\textit{Witt algebras}, respectively.

\begin{theorem}\label{zass}
For a Lie algebra of Witt type $W_R$, there is isomorphism of commutative algebras: 
$$
Der_{\frac{1}{2}} (W_R) \simeq K[\textstyle{\frac{1}{2}}R] ,
$$
where $\frac{1}{2}R = \set{\gamma\in R}{2\gamma\in R}$.
\end{theorem}

\begin{remark}
It is easy to see that the set $\frac{1}{2}R$ forms a semigroup, so $K[\frac{1}{2}R]$ 
is a semigroup algebra.
Of course, when $R=G$, then $\frac{1}{2}R = G$ and it is a group algebra.
\end{remark}

\begin{proof}
$W_R$ decomposes into the direct sum of one-dimensional root spaces $Ke_\alpha$, 
$\alpha\in R$, with respect to the action of the semisimple element $e_0$.
By considerations above, the space of $\frac{1}{2}$-derivations of $W_R$
is decomposed as the direct sum of one-dimensional root spaces with respect to 
the induced $e_0$-action. Any root of the latter action is the sum of two roots from $R$.
If an $\frac{1}{2}$-derivation $D$ of $W_R$ lies in the root space with the root 
$\gamma\in R + R$, then
$$
[e_0,D(e_\alpha)] - D([e_0,e_\alpha]) = \gamma D(e_\alpha) ,
$$
for any $\alpha\in R$, what implies that $D(e_\alpha)$ lies in the root space with the 
root $\alpha + \gamma$, i.e., belongs to $Ke_{\alpha + \gamma}$ if $\alpha + \gamma\in R$,
and vanishes otherwise. Assuming 
$$
D(e_\alpha) = \begin{cases}
\lambda_\alpha e_{\alpha + \gamma} , & \alpha + \gamma \in R      \\
0 ,                                  & \alpha + \gamma \notin R ,
\end{cases}
$$
for some $\lambda_\alpha\in K$,
the condition that $D$ is an $\frac{1}{2}$-derivation, written for the generic basis pair 
$e_\alpha, e_\beta$, is equivalent to:
\begin{equation}\label{lambda}
\lambda_{\alpha + \beta} = \begin{cases}
  \frac{\lambda_\alpha}{2} (1 - \frac{\gamma}{\beta - \alpha})
+ \frac{\lambda_\beta}{2}  (1 + \frac{\gamma}{\beta - \alpha}) ,
& \alpha + \gamma, \beta + \gamma \in R   \\
\frac{\lambda_\alpha}{2} (1 - \frac{\gamma}{\beta - \alpha})   ,   
& \alpha + \gamma \in R, \> \beta + \gamma \notin R  \\
0 ,
& \alpha + \gamma, \beta + \gamma \notin R .
\end{cases}
\end{equation}
for any $\alpha,\beta \in R$ such that $\alpha \ne \beta$, 
$\alpha + \beta + \gamma \in R$

Suppose $\gamma\notin R$. Substituting into the second case of (\ref{lambda}) $\beta = 0$, 
we get $\lambda_\alpha = 0$ for any $\alpha\in R$ such that $\alpha + \gamma \in R$.
Hence $D = 0$.

Suppose $\gamma\in R$. Substituting into the first case of (\ref{lambda}) $\alpha = 0$, 
we get 
\begin{equation}\label{l0}
\lambda_\beta = \lambda_0
\end{equation}
for any $\beta\in R$ such that $\beta \ne \gamma$.
If $\gamma = 0$, this implies that $D$ is proportional to the identity map.
Let $2\gamma\notin R$. Note that if $R$ consists of only two roots $0$ and $\gamma$,
then $W_R$ is isomorphic to the two-dimensional nonabelian Lie algebra. 
For this algebra, the statement of the Theorem is verified by elementary computations, 
so we can exclude this possibility from consideration.
Thus, substituting into the second case of (\ref{lambda}) any $\alpha \ne 0, \gamma$ and
$\beta = \gamma$ and taking into account (\ref{l0}), we get $\lambda_0 = 0$ and $D=0$.

If $\gamma\ne 0$ and $2\gamma\in R$, then substituting into the first case of 
(\ref{lambda}) $\alpha = \gamma$ and $\beta = 3\gamma$ 
(at this point, the assumption that the characteristic of the 
ground field $\ne 2,3$ is essential) and taking into account (\ref{l0}), yields
$\lambda_\gamma = \lambda_0$, and $D = \lambda_0 D_\gamma$, where
$D_\gamma(e_\alpha) = e_{\alpha + \gamma}$ for each $\alpha\in R$.

We have proved that each $\frac{1}{2}$-derivation of $W_R$ is a linear combination of
$D_\gamma$'s.
On the other hand, it is easy to see that for each $\gamma\in R$ such that 
$2\gamma\in R$, $D_\gamma$ is an $\frac{1}{2}$-derivation of $W_R$, and they are
linearly independent for different $\gamma$'s.
\end{proof}

Passing in Theorem \ref{zass} to the different specializations of $R$, as 
specified above, we get:
\begin{gather}
Der_{\frac{1}{2}} (W_1(n)) \simeq K[x_1, \dots, x_n]/(x_1^p, \dots, x_n^p) \label{zass1} 
\\
Der_{\frac{1}{2}} (W_1)    \simeq K[t,t^{-1}]   \notag \\
Der_{\frac{1}{2}} (W_1^+)  \simeq K[t] .        \notag
\end{gather}

All these three isomorphisms could be considered also as a manifestation of 
another general observation. Let $A$ be a commutative associative algebra and $\partial$ 
its derivation. Then we may consider a Lie algebra, denoted as $A\partial$, 
of derivations of $A$ of the form $a\partial$, $a\in A$. The Lie bracket is determined
by the formula 
$$
[a\partial, b\partial] = \big(a\partial(b) - b\partial(a)\big)\partial
$$
for $a,b\in A$.

\begin{proposition}\label{prop}
$Der_{\frac{1}{2}} (A\partial)$ contains a subalgebra isomorphic to $A$.
\end{proposition}

\begin{proof}
To any $u\in A$ we may associate a linear map $D_u: A\partial \to A\partial$ by
the rule $D_u(a\partial) = ua\partial$. A straightforward computation shows that $D_u$ is an
$\frac{1}{2}$-derivation of $A\partial$ for any $u\in A$.
\end{proof}

The $3$-dimensional algebra $sl(2)$ and the Zassenhaus algebra are characterized
among simple finite-dimensional Lie algebras in various interesting ways: 
these are algebras having a subalgebra of codimension $1$, 
algebras having a maximal solvable subalgebra, 
algebras with given properties of the lattice of subalgebras, etc.
Isomorphism (\ref{zass1}), together with the results of Filippov and Leger--Luks
cited above, adds another characterization to this list: these are the only 
finite-dimensional simple Lie algebras having nontrivial $\delta$-derivations
(for $\delta = -1$ in the case of $sl(2)$ and for $\delta = \frac{1}{2}$ in the case of 
the Zassenhaus algebra).

The Zassenhaus algebra admits many different realizations, and writing its 
$\frac{1}{2}$-derivations in these realizations may be beneficial for 
elucidating some interesting points. Recall that the \textit{divided powers algebra} $O_1(n)$ is a $p^n$-dimensional associative
commutative algebra with basis $\set{x^i}{0 \le i < p^n}$ and multiplication
$$
x^i x^j = \binom{i+j}{j} x^{i+j} .
$$
It is isomorphic to the reduced polynomial algebra 
$O_n = K[x_1, \dots, x_n]/(x_1^p, \dots, x_n^p)$.
$W_1(n)$ can be viewed as a Lie algebra of derivations $O_1(n)\partial$, where 
$\partial(x^i) = x^{i-1}$.
In such realization, it possesses a basis 
$\set{e_i = x^{i+1} \partial}{-1 \le i \le p^n - 2}$ with multiplication
$$
[e_i,e_j] = \Big(\binom{i+j+1}{j} - \binom{i+j+1}{i} \Big) e_{i+j} .
$$
Though in the general case the formulas expressing the concrete $\frac{1}{2}$-derivation
in terms of this basis appear to be cumbersome, in the case $n=1$ the basis of 
$Der_{\frac{1}{2}}(W_1(1))$ could be written as $\{1, D, \dots, D^{p-1}\}$, where
$D$ is the nilpotent map of the following neat form:
\begin{equation}\label{d}
D (e_i) = \begin{cases}
(i+2)e_{i + 1}, & -1 \le i < p - 2  \\
0             , & i = p - 2 .
\end{cases}
\end{equation}

As was noted by Kuznetsov (Remark at the very end of \cite{kuznetsov}),
$W_1(n)$ for $n>1$ also can be realized as deformation of the current Lie algebra
$W_1(1) \otimes O_1(n-1)$: $\{x,y\} = [x,y] + \Phi(x,y)$ for 
$x,y \in W_1(1) \otimes O_1(n-1)$, where the $2$-cocycle $\Phi$ is defined as
$$
\Phi (e_i \otimes a, e_j \otimes b) = \begin{cases} 
e_{p-2} \otimes (a \partial(b) - b \partial(a)), & i = j = -1      \\
0                                & \text{otherwise}                        
\end{cases}
$$
for $-1 \le i,j \le p-2$ and $a,b\in O_1(n-1)$. 
By Corollary \ref{corr}, 
$$
Der_{\frac{1}{2}} (W_1(1) \otimes O_1(n-1)) \simeq 
Der_{\frac{1}{2}} (W_1(1)) \otimes O_1(n-1) ,
$$
each $\frac{1}{2}$-derivation of $W_1(1)\otimes O(n-1)$ being 
the linear span of elements of the form 
\begin{equation}\label{tens}
D \otimes R_u
\end{equation}
for $D\in Der_{\frac{1}{2}}(W_1(1))$ and $u \in O_1(n-1)$.
According to the general principle, invariants
of algebras can only decrease under deformation (this easily could be made more precise
in the case of $\delta$-derivations following the general format of the Gerstenhaber's
deformation theory, but we will not delve into it), so $Der_{\frac{1}{2}} (W_1(n))$
is ``less or equal'' than 
$Der_{\frac{1}{2}} (W_1(1)) \otimes O_1(n-1)$. The comparison of dimensions ensures
that they are, in fact, equal. This could also be explained by the simple fact that 
for any $D\in Der_{\frac{1}{2}}(W_1(1))$ and $u\in O_1(n-1)$,
$$
\Phi((D\otimes R_u)(W_1(1)\otimes O_1(n-1), W_1(1)\otimes O_1(n-1))) = 0
$$
and
$$
(D\otimes R_u)(\Phi(W_1(1)\otimes O_1(n-1), W_1(1)\otimes O_1(n-1))) = 0 ,
$$
so each $\frac{1}{2}$-derivation of $W_1(1) \otimes O_1(n-1)$ can be lifted,
in a trivial way, to those of $W_1(n)$.

In \cite[\S 3]{filippov-2}, Filippov asked the question: is it true that the commutative
algebra of $\frac{1}{2}$-derivations of a prime Lie algebra does not contain divisors
of zero? 
According to (\ref{zass1}), Zassenhaus algebra provides a negative
answer to this question. In fact, for Zassenhaus algebra the situation is, in a sense,
opposite to the situation without divisors of zero: 
the algebra of $\frac{1}{2}$-derivations is a local one, i.e., its radical, 
consisting of nilpotent elements, is of codimension $1$. The concrete examples
of nilpotent $\frac{1}{2}$-derivations are provided by the map (\ref{d}) and its powers
in the case of the Witt algebra, and, more generally, by the maps of the form 
(\ref{tens}) in the case of the Zassenhaus algebra.

Further examples providing negative answer to the Filippov's question are 
semisimple Lie algebras having $W_1(n) \otimes O_m$ as the socle.
Indeed, semisimple Lie algebras of the form 
$W_1(n) \otimes O_m + 1\otimes \mathcal D$, where $\mathcal D$ is a subalgebra 
of $Der(O_m)$ such that $O_m$ does not have $\mathcal D$-invariant ideals, are prime.
It is possible to prove that under some conditions imposed on $L$, $A$ (like, for example, 
in Corollary \ref{corr}), and on $\mathcal D \subseteq Der(A)$, 
$$
Der_\delta (L\otimes A + 1\otimes \mathcal D) \simeq 
Der_\delta(L) \otimes A^{\mathcal D} .
$$
In particular, as the absence of $\mathcal D$-invariant ideals in $O_m$ implies 
$O_m^{\mathcal D} = K1$, we have:

$$
Der_{\frac{1}{2}} (W_1(n) \otimes O_m + 1\otimes \mathcal D) \simeq 
Der_{\frac{1}{2}} (W_1(n)) \simeq O_n .
$$

According to the classical Block's theorem about the structure of modular 
finite-di\-men\-si\-o\-nal semisimple Lie algebras (\cite[Theorem 9.3]{block}), this, essentially,
exhausts all possible examples of finite-dimensional prime Lie algebras whose
$\frac{1}{2}$-derivations have divisors of zero.

The same constructions provide also infinite-dimensional examples of prime
Lie algebras whose $\frac{1}{2}$-derivations have divisors of zero.
For example, the Lie algebra $W_{\mathbb Q}$, where $\mathbb Q$ is the additive group
of rationals, is simple (see, for example, \cite[Chapter 10, Theorem 3.1]{as-book}).
As $\mathbb Q$ is a group with torsion, the group algebra $K[\mathbb Q]$ contains
divisors of zero. Another example: the Lie algebra $W_1(1) \otimes K[x] + 1\otimes \frac{d}{dx}$
is prime, and it is easy to see, by the same reasoning as for finite-dimensional modular 
semisimple Lie algebras above, that its algebra of 
$\frac{1}{2}$-derivations is isomorphic to $K[x]/(x^p)$.

\section{Non-semigroup gradings}\label{semigr}

A \textit{grading} of a Lie algebra $L$ is the direct sum decomposition 
$L = \bigoplus_{\alpha\in G} L_\alpha$ into subspaces indexed by a set $G$ 
such that for each $\alpha, \beta \in G$ either $[L_\alpha, L_\beta] = 0$
or there is $\gamma(\alpha,\beta)\in G$ such that 
$[L_\alpha, L_\beta] \subseteq L_{\gamma(\alpha,\beta)}$.
All ``interesting'' gradings appearing on practice (like Cartan decompositions,
gradings induced by automorphisms of finite order, 
$\mathbb Z$-gradings associated with filtrations, etc.) are 
\textit{group (or semigroup) gradings},
i.e., $G$ could be embedded into an (additively written) abelian group (or semigroup) such 
that $\gamma(\alpha,\beta) = \alpha + \beta$ for any pair $\alpha,\beta\in G$ for which
$\gamma(\alpha,\beta)$ exists.

However, a priori the existence of such embedding is not obvious, and the natural
question arises whether each grading of a finite-dimensional Lie algebra is a semigroup
grading. In \cite{elduque-1} Elduque constructed a quite unexpected example, refuting
a two-decades old claim by Patera and Zassenhaus \cite[Theorem 1(a)]{zass}. 
Further examples constructed in \cite{elduque-2} allowed to think that non-semigroup 
gradings are quite common.
However, they seem to be produced by ad-hoc trial and error, and no systematic 
method of constructing such examples was known till now.

Suppose $\mathcal D$ is a set of commuting $\delta$-derivations of a Lie algebra $L$
over an algebraically closed field $K$. Consider the standard root space decomposition with
respect to the set $\mathcal D$: 
\begin{equation}\label{root}
L = \bigoplus_{\lambda\in \mathcal D^*} L_\lambda .
\end{equation}
The same simple inductive reasoning based on the binomial formula and used in the proof of the usual multiplicative 
properties of root spaces with respect to ordinary derivations (see, for example,
\cite[Chapter III, \S 2]{jacobson}), can be used to prove that
\begin{equation}\label{mult}
[L_\lambda, L_\mu] \subseteq L_{\delta(\lambda + \mu)}
\end{equation}
for any $\lambda, \mu \in \mathcal D^*$.

When (\ref{root}) will be a semigroup grading and when it will not? Suppose for a moment
that all products between root spaces are nonzero. Then the set
of elements $\lambda\in \mathcal D^*$ such that $L_\lambda \ne 0$, with operation
\begin{equation}\label{circ}
\lambda \circ \mu = \delta(\lambda + \mu)
\end{equation}
should be a semigroup. Associativity of $\circ$, due to (\ref{mult}),  
is equivalent to 
$$
(\delta^2 - \delta)(\lambda - \mu) = 0
$$
for any $\lambda, \mu$.
Clearly, for $\delta\ne 0,1$ this implies that any root space decomposition 
(\ref{root}) with more than one root space would be a non-semigroup grading.

The problem is that some products of root spaces may vanish, so if the ratio
of the number of root spaces to the number of pairs of root spaces with zero product 
is ``small'', the operation (\ref{circ}) on the remaining pairs could be not enough to violate 
associativity.
Accordingly, under the natural assumption of having ``enough'' root spaces with nonzero 
product, we are guaranteed to get a non-semigroup grading, as the following elementary
proposition shows.

\begin{proposition}\label{root1}
Let $L$ be a finite-dimensional Lie algebra over an algebraically closed field,
$\delta \ne 0,1$, and $\mathcal D$ is a set of commuting $\delta$-derivations of $L$. 
Then the root space decomposition (\ref{root}) with respect to $\mathcal D$
is a non-semigroup grading of $L$ if one of the following holds:
\begin{enumerate}
\item
there are three pairwise distinct roots 
$\lambda, \mu, \eta \in \mathcal D^*$ such that $[[L_\lambda, L_\mu], L_\eta] \ne 0$;
\item
there are two distinct roots $\lambda, \mu \in \mathcal D^*$ such that 
$[[L_\lambda, L_\lambda],L_\mu] \ne 0$.
\end{enumerate}
\end{proposition}

\begin{proof}
(i)
Suppose that (\ref{root}) is a semigroup grading.
The Jacobi identity and the condition of Proposition imply that at least one of the 
triple products $[[L_\eta, L_\lambda], L_\mu]$ and $[[L_\mu, L_\eta], L_\lambda]$ 
does not vanish too. Suppose, without loss of generality, that 
$[L_\lambda, [L_\mu, L_\eta]] \ne 0$. Then both expressions 
$(\lambda \circ \mu) \circ \eta$ and $\lambda \circ (\mu \circ \eta)$ exist, hence
they are equal, what implies, as noted above, $\lambda = \eta$.

(ii)
Similarly: the Jacobi identity implies $[[L_\lambda, L_\mu], L_\lambda] \ne 0$
and the same reasoning applies to derive $\lambda = \mu$, a contradiction.
\end{proof}

On the other hand, it is obvious that if there is only one root space,
for example, if all $\delta$-derivations in $\mathcal D$ are nilpotent, 
then the grading collapses to one summand and we get nothing interesting in this way.
 
Proposition \ref{root1} provides a way to construct non-semigroup gradings of Lie
algebras in a systematic way. Some of the examples constructed by Elduque
could also be recovered using this scheme. For example, consider the $4$-dimensional Lie algebra 
from \cite[\S 1]{elduque-2}. This is an algebra with basis $\{a, u, v, w\}$ and 
multiplication table: 
$$
[a,u]=u, \quad [a,v]=w, \quad [a,w]=v ,
$$
all other products between basic elements being 
zero. It has the following $(-1)$-derivation:
$$
a \mapsto 0, \quad u \mapsto 0, \quad v \mapsto -w, \quad w \mapsto v .
$$
The root space decomposition with respect to this $(-1)$-derivation is:
$$
(Ka \oplus Ku) \oplus K(v+iw) \oplus K(v-iw) ,
$$
the direct summands correspond to eigenvalues $0$, $i$ and $-i$. We have:
$$
0 \circ 0 = 0, \quad 0 \circ i = -i, \quad 0 \circ (-i) = i .
$$
This is exactly the operation on the set of $3$ elements obtained in 
\cite{elduque-2} (though the grading itself is a bit different), which is shown there not to be embeddable 
in any semigroup.

On the other hand, not all examples of non-semigroup gradings could be obtained
in this way.
For example, for each $\delta\ne 0$, the space of $\delta$-derivations of the
$16$-dimensional nilpotent Lie algebra from \cite{elduque-1}, 
represented as matrices in the given basis of the algebra, has a basis which could be
divided into $2$ or $3$ non-intersecting sets: all matrices in the first set
are strictly upper-triangular; for each pair $X$, $Y$ of matrices from the second set,
$XY = 0$ and $XBY = 0$ holds for any matrix $B$ from the basis; 
and, in the case $\delta = \frac{1}{2}$, the identity matrix. 
This is verified on computer (see Appendix).
Consequently, each $\delta$-derivation of that algebra can be represented as the sum of 
a (possibly zero) multiplication by a field element and a nilpotent $\delta$-derivation. 

Note also that, unfortunately, we cannot in this way get a (negative) answer
to the question from \cite{elduque-1}: is every grading of a \emph{simple}
finite-dimensional Lie algebra over an algebraically closed field of characteristic zero a semigroup grading?
For, as mentioned in the previous section, either by combined results of Theorems FA, FB 
and FC, or by Theorem LL, among such algebras only $sl(2)$ has nontrivial 
$\delta$-derivations, and, obviously, every grading of $sl(2)$ is a semigroup grading.

In conclusion of this section, let us note another fact which, though, not
directly related to non-semigroup gradings, also follows from elementary considerations
related to the root space decomposition with respect to $\delta$-derivations:

\begin{proposition}
For any non-nilpotent $\delta$-derivation of a finite-dimensional perfect Lie algebra over 
the field of complex numbers, $\frac{1}{2} \le |\delta| \le 1$ and $\delta$ is algebraic.
\end{proposition}

\begin{proof}
Let $L$ be such a Lie algebra and $D$ is its $\delta$-derivation.
From (\ref{root}) (where $\mathcal D$ is assumed to consist of a single element 
$D$ and the set of corresponding roots is denoted as $R \subset \mathbb C$), 
it follows that
$$
L = [L,L] = \sum_{\lambda, \mu \in R} [L_\lambda, L_\mu] .
$$
Hence, in view of (\ref{mult}), for each root $\eta\in R$, 
$L_\eta \subseteq \sum_{\eta = \delta(\lambda + \mu)} L_{\delta(\lambda + \mu)}$.
In particular, for each $\eta\in R$ the set over which the last summation is performed is 
not empty, i.e., there are $\lambda, \mu\in R$ such that 
\begin{equation}\label{sum}
\eta = \delta(\lambda + \mu) .
\end{equation}
By induction, for any $\eta\in R$ and $n \in \mathbb N$, there are 
$\lambda_1, \dots, \lambda_{2^{n+1}}\in R$ such that 
$\eta = \delta^n (\lambda_1 + \dots + \lambda_{2^{n+1}})$. Then, denoting 
$M = \max\limits_{\lambda\in R} |\lambda|$, we have $M \le |\delta|^n 2^{n+1} M$ for any
$n\in \mathbb N$, and either 
$$
|\delta| \ge \lim_{n\to\infty} \frac{1}{\sqrt[n]{2^{n+1}}} = \frac{1}{2} ,
$$
or $M = 0$. In the latter case the set of roots consists of a single root $0$, whence
$D$ is nilpotent.

Similarly, denoting 
$$
N = 
\min_{\substack{k_1\lambda_1 + \dots + k_m\lambda_m \ne 0 \\ 
                \lambda_i\in R, \> k_i\in \mathbb Z}} |k_1\lambda_1 + \dots + k_m\lambda_m|
$$
(as $R$ is finite and does not consist of a single element $0$, this minimum exists), 
we have $|\eta| \ge |\delta|^n N$ for any nonzero $\eta\in R$ and $n\in \mathbb N$, 
whence $|\delta| \le 1$. 

Finally, writing the condition (\ref{sum}) for each root $\eta\in R$, we get a homogeneous
system of $|R|$ equations with coefficients depending linearly, over $\mathbb Z$, 
on $\delta$, and in $|R|$ unknowns.
This system has a nonzero solution, hence its determinant equal to zero, 
what provides a nonzero polynomial with integer coefficients vanishing on $\delta$.
\end{proof}

\section{$\delta$-derivations and $\delta$-superderivations of prime Lie superalgebras}
\label{super}

$\delta$-derivations of superalgebras are defined, as in the case of ordinary algebras,
by the condition (\ref{delta-der}).
One may argue that the more natural approach would be to generalize the corresponding
super notion.
Recall that a homogeneous linear map $D: A \to A$
is called a \textit{superderivation of a superalgebra} $A = A_0 \oplus A_1$ if
$$
D(ab) = D(a)b + (-1)^{\deg(D)\deg(a)} aD(b)
$$
for any two homogeneous elements $a,b\in A$.
The \textit{supercentroid} of $A$ is the set of all homogeneous linear maps 
$\chi: A \to A$ such that 
$$
\chi(ab) = \chi(a)b = (-1)^{\deg(\chi)\deg(a)} a\chi(b)
$$
for any two homogeneous elements $a,b\in A$.
Accordingly, let us call a homogeneous linear map $D$ a 
\textit{$\delta$-superderivation}
of $A$ if
$$
D(ab) = \delta D(a)b + \delta (-1)^{\deg(D)\deg(a)} aD(b)
$$
for any two homogeneous elements $a,b\in A$. Like in the ordinary case, this 
generalizes superderivations (for $\delta=1$) and elements of supercentroid (for $\delta = \frac{1}{2}$).

Here we will consider $\delta$-de\-ri\-va\-tions of Lie superalgebras, like in 
\cite{kayg}, as well as $\delta$-su\-per\-de\-ri\-va\-tions.
These two notions overlap: even $\delta$-superderivations are the same as 
homogeneous (i.e., preserving the $\mathbb Z_2$-grading) $\delta$-derivations. 

Though the mathematics below is pretty much elementary (modulo existing results), we call the reader 
to distinct carefully between ordinary, $\mathbb Z_2$-graded and super variants
of different constructions, as we pass freely back and forth between them.

Let $A$ and $B$ be two superalgebras. Their tensor product $A \otimes B$
as of ordinary algebras (i.e., with multiplication defined similarly to (\ref{curr}), 
as for the current Lie algebras), will, naturally, carry a superalgebra structure, 
with the zero component 
$$
(A \otimes B)_0 = (A_0 \otimes B_0) \oplus (A_1 \otimes B_1) .
$$
If $G = G_0 \oplus G_1$, the Grassmann superalgebra in the countable number of
odd variables, $(A \otimes G)_0$ is nothing but the \textit{Grassmann envelope of $A$},
denoted as $G(A)$ -- a construction proved to be very useful in many questions pertained 
to varieties of (ordinary) algebras. 
It is well known (and easy to see) that for an arbitrary superalgebra $A$,
$G(A)$ satisfies some identity $w$ if and only if $A$ satisfies the corresponding 
superidentity (obtained from $w$ by appropriately injecting
signs). In particular, $L$ is a Lie superalgebra if and only if $G(L)$ is a Lie algebra
(on the other hand, the whole $L \otimes G$ is, in general, not a Lie superalgebra).

Though it is possible to develop a super or $\mathbb Z_2$-graded modifications of the 
technique from \cite{low}, in general it seems to be fairly 
difficult to say something general about $\delta$-derivations of $G(L)$ in terms of $L$.
However, the following elementary observations take place:

\begin{lemma}\label{der}
Let $A$ and $B$ be two superalgebras.
\begin{enumerate}
\item
If $D$ is a $\delta$-derivation of $A$, then the map 
$\widehat D: (A\otimes B)_0 \to A\otimes B$ defined as
\begin{equation*}
a_0 \otimes b_0 + a_1 \otimes b_1 \mapsto D(a_0) \otimes b_0 + D(a_1) \otimes b_1
\end{equation*}
for $a_0\in A_0$, $a_1\in A_1$, $b_0\in B_0$, $b_1\in B_1$, is a 
$\delta$-derivation of $(A \otimes B)_0$ with values in $A\otimes B$.
\item
If $D$ is a $\delta$-superderivation of $A$, and $\chi$ is an element of the supercentroid
of $B$ such that $\deg(D) = \deg(\chi)$, then the map 
$\widehat D: (A\otimes B)_0 \to (A\otimes B)_0$ defined as
$$
a_0 \otimes b_0 + a_1 \otimes b_1 \mapsto D(a_0) \otimes \chi(b_0) 
                                        + D(a_1) \otimes \chi(b_1)
$$
for $a_0\in A_0$, $a_1\in A_1$, $b_0\in B_0$, $b_1\in B_1$, is a 
$\delta$-derivation of $(A \otimes B)_0$.
\end{enumerate}
\end{lemma}

\begin{proof}
Direct verification.
\end{proof}

And similarly:

\begin{lemma}\label{yoyod}
Let $L = L_0 \oplus L_1$ be a Lie superalgebra, $\form$ a nondegenerate 
supersymmetric invariant bilinear form on $L$, and $f: G \to K$ a nonzero linear form. 
Then the bilinear map $\form: G(L) \times G(L) \to K$ defined as
\begin{multline*}
(x_0 \otimes g_0 + x_1 \otimes g_1, x_0^\prime \otimes g_0^\prime + x_1^\prime \otimes g_1^\prime)
\\= (x_0, x_0^\prime) \otimes f(g_0g_0^\prime) 
  + (x_0, x_1^\prime) \otimes f(g_0g_1^\prime) 
  + (x_1, x_0^\prime) \otimes f(g_1g_0^\prime) 
  + (x_1, x_1^\prime) \otimes f(g_1g_1^\prime)
\end{multline*}
for $x_0,x_0^\prime \in L_0$, $x_1,x_1^\prime \in L_1$, $g_0,g_0^\prime \in G_0$, 
$g_1,g_1^\prime \in G_1$, is a nondegenerate symmetric 
invariant bilinear form on $G(L)$.
\end{lemma}

\begin{proof}
Direct verification.
\end{proof}

A thorough analysis, sometimes with slight modifications, of some of the Filippov's 
arguments allows to establish a bit more than actually was stated in his papers.

\begin{lemma}\label{ideal}
If $I$ is a nonzero ideal of a prime Lie superalgebra $L$, $\delta \ne 0$,
and $D$ is a nonzero $\delta$-derivation or $\delta$-superderivation of $L$, 
then $D(I) \ne 0$.
\end{lemma}

\begin{proof}
This was proved in \cite[Lemma 3]{filippov-1} in the case of Lie algebras, and exactly 
the same proof is valid in the case of superalgebras.
\end{proof}

Recall that the \textit{standard identity of degree $5$} is the Lie identity of the
form
\begin{equation*}
\sum_{\sigma\in S_4}
(-1)^\sigma [[[[y, x_{\sigma(1)}], x_{\sigma(2)}], x_{\sigma(3)}], x_{\sigma(4)}] = 0
\end{equation*}
(summation is performed over all elements of the symmetric group $S_4$).
Its super analog (with appropriately injected signs) will be called the
\textit{standard superidentity of degree $5$}.
By $s_4(L)$ for a Lie (super)algebra $L$ we will understand the ideal of 
$L$ consisting of the linear spans of the left hand side of the standard
(super)identity of degree $5$ for all $x_1, x_2, x_3, x_4, y\in L$.

\begin{lemma}\label{filippov}
Let $L$ be a Lie algebra, $M$ an $L$-module, and $D$ a $\delta$-derivation
of $L$ with values in $M$, $\delta\ne 0,1,\frac{1}{2}$. Then $s_4(L) \subseteq Ker\,D$.
\end{lemma}

\begin{proof}
Consider the semidirect sum $\mathscr L = L \oplus M$, with the Lie bracket defined in 
the usual way: $[m,x] = x \bullet m$ for $x \in L, m\in M$, and $[M,M] = 0$.
Extend $D$ to a linear map $\mathscr D: \mathscr L \to \mathscr L$ by putting 
$\mathscr D (M) = 0$. 
It is obvious that the so extended map is a $\delta$-derivation of $\mathscr L$.
By \cite[Theorem 1]{filippov-2}, $s_4(\mathscr L) \subseteq Ker\,\mathscr D$.
It is obvious that $s_4(\mathscr L) = s_4(L) \oplus M^\prime$ for some subspace
$M^\prime \subseteq M$. On the other hand, $Ker\,\mathscr D = Ker\,D \oplus M$,
and the desired statement follows.
\end{proof}

\begin{lemma}\label{solv}
A Lie algebra satisfying the standard identity of degree $5$ and 
having a nonzero $\delta$-derivation, $\delta \ne -1,0,\frac{1}{2},1$, is solvable.
\end{lemma}

\begin{proof}
In the proof of Theorem FA in \cite{filippov-2}, the condition of primeness of an
algebra is used only twice: first to make use of \cite[Lemma 3]{filippov-2} claiming
that a prime Lie algebra having a nonzero $\delta$-derivation, 
$\delta \ne 0,\frac{1}{2},1$, satisfies the standard identity of degree $5$ and a certain
its consequence, and then at the very end, to obtain contradiction with solvability.
Consequently, that proof, almost verbatim, will serve as the proof of the Lemma.
\end{proof}

\begin{lemma}\label{symm}
The space of $\frac{1}{2}$-derivations of a perfect centerless Lie algebra having a 
nondegenerate symmetric invariant bilinear form, coincides with the centroid of the 
algebra.
\end{lemma}

\begin{proof}
In the proof of \cite[Theorem 6]{filippov-1}, the condition of primeness of an 
algebra $L$ is used only at the very end, to deduce that 
$$
\set{x\in L}{[[L,L],x] = 0} = 0 .
$$
But for a perfect and centerless Lie algebra, the latter condition will be 
trivially satisfied as well.
\end{proof}

All this together allows to establish super analogs of Theorems FA and FB:

\begin{theorem}\label{prime}
A prime Lie superalgebra does not have nonzero $\delta$-de\-ri\-va\-tions and 
nonzero $\delta$-su\-per\-de\-ri\-va\-tions if 
$\delta \ne -1, 0, \frac{1}{2}, 1$.
\end{theorem}

\begin{proof}
Consider first the case of $\delta$-derivations.
Let $D$ be a nonzero $\delta$-derivation of a prime Lie superalgebra $L$.
By Lemma \ref{der}(i), $\widehat D$ is a nonzero $\delta$-derivation of $G(L)$
with values in $L\otimes G$. By Lemma \ref{filippov}, 
\begin{equation}\label{yoyo}
s_4(G(L)) \subseteq Ker(\widehat D) ,
\end{equation} 
Obviously, 
\begin{equation}\label{yoyo1}
s_4(G(L)) = G(s_4(L)) .
\end{equation}
On the other hand, 
\begin{equation}\label{yoyo2}
Ker(\widehat D) = (Ker(D|_{L_0}) \otimes G_0) \oplus (Ker(D|_{L_1}) \otimes G_1) .
\end{equation}
It follows from (\ref{yoyo}), (\ref{yoyo1}) and (\ref{yoyo2}) that
$$
s_4(L) = s_4(L)_0 \oplus s_4(L)_1 \subseteq Ker(D|_{L_0}) \oplus Ker(D|_{L_1}) 
\subseteq Ker\,D ,
$$
where $s_4(L) = s_4(L)_0 \oplus s_4(L)_1$ is decomposition of the Lie superalgebra
$s_4(L)$ into the even and odd parts.
By Lemma \ref{ideal}, $s_4(L) = 0$, in other words, $L$ satisfies the standard
superidentity of degree $5$, and $G(L)$ satisfies the standard identity
of degree $5$. By Lemma \ref{solv}, $G(L)$ is solvable and hence $L$ is solvable, 
a contradiction.

Now let $D$ be a $\delta$-superderivation of $L$. Take as $\chi$ the left multiplication 
by any homogeneous element of $G$ whose parity coincides with the parity of $D$ and
apply Lemma \ref{der}(ii). $\widehat D$ is a nonzero $\delta$-derivation of $G(L)$, and
the rest of reasoning is the same as in the case of $\delta$-derivations.
\end{proof}

\begin{theorem}\label{onetwo}
The space of $\frac{1}{2}$-derivations (respectively, $\frac{1}{2}$-superderivations)
of a perfect centerless Lie superalgebra having a nondegenerate supersymmetric 
invariant bilinear form, coincides with the centroid (respectively, supercentroid)
of the superalgebra.
\end{theorem}

\begin{proof}
Again, consider first the case of $\frac{1}{2}$-derivations.
Let $L$ be such a superalgebra, and $D$ is its $\frac{1}{2}$-derivation.
By Lemma \ref{der}(i), $\widehat D$ is an $\frac{1}{2}$-derivation of $G(L)$
with values in $L \otimes G$.
We employ the same elementary trick as in the proof of Lemma \ref{filippov} and consider
the semidirect sum $\mathscr L = G(L) \oplus (L \otimes G)$, where the Lie 
bracket between $G(L)$ and $L\otimes G$ is determined by action of the former on the 
latter, and the space $L\otimes G$ is considered as an abelian subalgebra.
It is obvious that this semidirect sum is perfect and centerless.
We can also furnish it with the symmetric bilinear form: on $G(L)$ it is defined
as in Lemma \ref{yoyod}, between elements of $G(L)$ and $L\otimes G$ it is defined, 
essentially, in the same way: 
$$
(x\otimes g, x^\prime \otimes g^\prime) = (x,x^\prime) f(gg^\prime)
$$
for homogeneous elements $x,x^\prime\in L$, $g,g^\prime \in G$, and on $L\otimes G$
the form vanishes. The same elementary calculation that verifies the validity of
Lemma \ref{yoyod}, verifies that this form is invariant and nondegenerate.

We also extend $\widehat D$ from $L(G)$ to $\mathscr L$ as in the proof of Lemma 
\ref{filippov}, and apply to the resulting $\frac{1}{2}$-derivation of $\mathscr L$
Lemma \ref{symm} to get:
$$
D(x_0) \otimes g_0 + D(x_1) \otimes g_1 = \widehat\chi (x_0 \otimes g_0 + x_1 \otimes g_1)
$$
for any $x_0\in L_0$, $x_1\in L_1$, $g_0\in G_0$, $g_1\in G_1$, and for a certain 
$\widehat\chi$ belonging to the centroid of $\mathscr L$.
This immediately implies that $D$ belongs to the centroid of $L$.

The case of $\frac{1}{2}$-superderivations is similar, but simpler: we appeal to 
Lemma \ref{der}(ii) (where $\chi$ is taken, as in the respective part of the proof of
Theorem \ref{prime}, as the left multiplication by a homogeneous element $g\in G$ of the appropriate 
parity), directly to Lemma \ref{yoyod}, and to 
Lemma \ref{symm}, to deduce that for every $\frac{1}{2}$-superderivation $D$ of $L$,
$$
D(x_0) \otimes gg_0 + D(x_1) \otimes gg_1 = 
\widehat\chi (x_0 \otimes g_0 + x_1 \otimes g_1)
$$
holds for any $x_0\in L_0$, $x_1\in L_1$, $g_0\in G_0$, $g_1\in G_1$, and 
for a certain $\widehat\chi$ belonging to the centroid of $G(L)$.
This immediately implies that $D$ belongs to the supercentroid of $L$.
\end{proof}

Theorem \ref{prime} (together with its proof) and Theorem \ref{onetwo} allow to 
streamline most of the proofs from \cite{kayg} of the absence of nontrivial 
$\delta$-derivations of classical Lie superalgebras. For example, 
when dealing with $(-1)$-derivations, the condition that
a Lie superalgebra satisfies the standard superidentity of degree $5$ allows to exclude
most of the cases, and when dealing with $\frac{1}{2}$-derivations, Theorem \ref{onetwo}
allows to settle immediately the cases of superalgebras 
$A(m,n)$ with $m \ne n$, $B(m,n)$, $C(n)$,
$D(m,n)$ with $m - n \ne 1$, $F(4)$ and $G(3)$, as follows from \cite[\S 2.3.4]{kac}.

Note that a similar approach can be used to obtain a super analog of the Filippov's result
about absence of nonzero $\delta$-derivations of prime associative algebras
for $\delta \ne 0,\frac{1}{2},1$ (\cite[Theorem 1]{filippov-ass}). 
In that case, instead of taking the Grassmann envelope,
we may take a similar tensor product construction with any centrally closed associative
superalgebra (for example, a free associative superalgebra). Then one may extend in the
straightforward way notions and results from \cite{emo} pertaining generalized 
(Martindale) centroid and primeness of the tensor product of prime algebras to the
super case, and utilize Lemma \ref{der} to reduce the super case to the ordinary one.

\section{Further questions}

\subsection{}\label{skryabin}

The isomorphisms (\ref{zass1}) in \S \ref{sect-zass}, as well as Proposition \ref{prop}, suggest that 
$\delta$-derivations of a Lie algebra of derivations of a commutative algebra are strongly
related to the underlying commutative algebra.
In this connection, it would be interesting to try to obtain $\delta$-derivational 
counterparts of the very general results \cite{skryabin} of Skryabin about the first
cohomology of Lie algebras of derivations of commutative rings.

\subsection{}

It was suggested by Dimitry Leites and his collaborators (see, for example, 
\cite[\S 1]{bg-leites} and \cite{leites} and references therein),
that in some circumstances the cohomology theory of modular Lie (super)algebras,
especially in characteristic $2$, is inadequate to describe the structural phenomena,
and they posed a question to devise a more ``ideologically correct'' cohomology.
Leites suggested further that $\delta$-derivations may have something to do with that.
Modular Lie superalgebras or not, it appears to be interesting to devise a cohomology 
theory which will encompass $\delta$-derivations.
All ``naive'' attempts to do that seemingly fail.

\subsection{}

Let $D$ be a $\delta$-derivation of a Lie algebra $L$. Suppose that one of the following 
holds: 
$D$ is nilpotent with index of nilpotency less than the (positive) characteristic of the 
ground field;
$D$ is nilpotent and the ground field is of characteristic $0$;
$L$ is finite-dimensional and the ground field is $\mathbb R$ or $\mathbb C$.
Then 
$$
\exp(\delta D) = 1 + \delta D + \frac{\delta^2 D^2}{2!} + \frac{\delta^3 D^3}{3!} + \dots: L \to L
$$ 
is a well-defined linear map, and
$$
\exp(D)([x,y]) = [\exp(\delta D)(x), \exp(\delta D)(y)]
$$
holds for any $x,y\in L$. This is proved exactly in the same way as in the classical
case of derivations for $\delta=1$ (see, for example, \cite[Chapter 1, \S 2]{jacobson}).

More generally, if $D$ is a quasiderivation of $L$ such that one of the assumptions
above holds both for $D$ and $F(D)$, then
$$
\exp(F(D))([x,y]) = [\exp(D)(x), \exp(D)(y)] 
$$ 
holds for any $x,y\in L$.

This suggests the following definition. Let us call a bijective linear map
$\varphi: L \to L$ a \textit{quasiautomorphism of $L$} if there exists a bijective 
linear map $\psi: L \to L$ such that 
$$
\psi([x,y]) = [\varphi(x),\varphi(y)]
$$ 
holds for any $x,y\in L$. Obviously, the set of all quasiautomorphisms of a Lie algebra
forms a group. To our knowledge, this notion was not studied yet.
May be it is interesting to undertake such a study.

\section*{Acknowledgement}

Thanks are due to late Hans Zassenhaus for supplying me, long time ago, 
with a reprint of his paper \cite{zass}. Though nowadays, with all these electronic 
accesses, this does not seem to be a big deal, it was such at 1989, 
when I was a student in a remote part of the former Soviet Union.
The letter from him containing the requested reprint with warm words of encouragement
would always be remembered with gratitude and nostalgia.

\section*{Appendix}

Here we describe a simple-minded GAP program, available as \newline
\texttt{http://justpasha.org/math/delta-der.gap}\footnote[2]{
Currently available as an ancillary file in the arXiv version.
}, 
for computation of $\delta$-derivations of a Lie (and, more general, anticommutative) algebra.
It was used to verify various claims made in this paper.

Let $A$ be an anticommutative algebra with basis $\{e_1, \dots, e_n\}$ defined
over a field $K$, and with multiplication table 
$e_ie_j = \sum_{k=1}^n C_{ij}^k e_k$, and let $\delta\in K$ and $D$ be a
$\delta$-derivation of $A$. Writing $D(e_i) = \sum_{j=1}^n d_{ij}e_j$,
for certain $d_{ij} \in K$,
the condition (\ref{delta-der}), written
for each pair of basic elements, is equivalent to the system of 
$\frac{n^2(n-1)}{2}$ linear homogeneous equations in $d_{ij}$:
\begin{equation}\label{sys}
\sum_{k=1}^n C_{ij}^k d_{kl} - \delta \sum_{k=1}^n C_{kj}^l d_{ik} 
                             + \delta \sum_{k=1}^n C_{ki}^l d_{jk} 
= 0
\end{equation}
for $1 \le i < j \le n$, $1 \le l \le n$.
So we just solve this system using the standard library routine.

When dealing with the $16$-dimensional Lie algebra from 
\cite{elduque-1} (as mentioned in \S \ref{semigr}), we want to compute $\delta$-derivations
for \textit{all} $\delta$, i.e., to solve a linear system with parameter.
As GAP (version 4.4.12 as of time of this writing) does not support transcendental
field extensions -- what would be a natural way to work with parameters --
we are cheating by using cyclotomic fields instead.
However, this cheating could be made rigorous by picking a cyclotomic extension of prime
degree (of course, any other field extension by an irreducible polynomial will do) larger 
than the highest possible power of a parameter involved in computation.
For example, if we deal with a $16$-dimensional algebra, the system (\ref{sys}) will be
of size $1920 \times 256$. As all coefficients of that system depend on the 
parameter $\delta$ linearly, any power of $\delta$ occurring in its solution will not 
exceed $256$, so the cyclotomic extension of degree $257$ would be enough.

\end{document}


\title{Erratum to "On $\delta$-derivations of Lie algebras and superalgebras"}
\author{Pasha Zusmanovich}
\address{
Department of Mathematics, Tallinn University of Technology, Ehitajate tee 5, 
Tallinn 19086, Estonia
}
\email{pasha.zusmanovich@ttu.ee} 
\date{July 15, 2013}

\maketitle

In \cite{delta}, an incorrect definition of the Grassmann envelope
of a Lie superalgebra was used. Namely, if $L = L_0 \oplus L_1$ is a Lie
superalgebra, and $G = G_0 \oplus G_1$ is the Grassmann algebra, the 
multiplication on $G(L) = (L_0 \otimes G_0) \oplus (L_1 \otimes G_1)$
can be defined in two ways; as for ordinary algebras:
$$
[x \otimes g, y \otimes h] = [x,y] \otimes gh ,
$$
and as for superalgebras:
$$
[x \otimes g, y \otimes h] = (-1)^{\deg(g)\deg(y)} [x,y] \otimes gh ,
$$
where $x,y$ and $g,h$ are homogeneous elements of $L$ and $G$, respectively. 
These two multiplications lead to what is called in \cite[p.~72]{berezin}
the Grassmann envelopes of the first and second kind. However, as justly 
remarked in the footnote at the same page of \cite{berezin}, the first definition 
contradicts the Sign Rule, and will lead, sooner or later, to confusion. 
\cite{delta} is an instance of such a confusion:
it is easy to see that the Grassmann envelope of the first kind is, generally, 
not a Lie algebra, what defeats its main utility,
at least in the context of \cite{delta}.

To remedy this, one should adopt the definition of the Grassmann envelope of 
the second kind, and inject signs at appropriate places. First, the map in 
Lemma 4.1(ii) should be defined as
$$
a_0 \otimes b_0 + a_1 \otimes b_1 \mapsto 
D(a_0) \otimes \chi(b_0) + (-1)^{\deg(D)} D(a_1) \otimes \chi(b_1) .
$$

Second, the symmetric bilinear form used in the proof of Theorem 4.8
should be defined as
$$
(x \otimes g, x^\prime \otimes g^\prime) = 
(-1)^{\deg(g)\deg(x^\prime)} (x,x^\prime) f(gg^\prime) .
$$

No modifications in the proofs are needed.

Thanks are due to Liangyun Chen, Yao Ma, Yongzheng Zhang and Keli Zheng 
for bringing this defect to my attention, and to Dimitry Leites for getting
me acquainted with the 2nd edition of \cite{berezin}.